\scrollmode
\documentclass[12pt,reqno]{amsart}
\usepackage{amssymb}
\usepackage[all]{xy}

\theoremstyle{plain}
\newtheorem{prop}{Proposition}
\newtheorem{thm}[prop]{Theorem}
\newtheorem{cor}[prop]{Corollary}
\newtheorem{lem}[prop]{Lemma}

\theoremstyle{definition}
\newtheorem{defi}[prop]{Definition}

\theoremstyle{remark}

\newtheorem{rem}[prop]{Remark}
\newtheorem{example}[prop]{Example}

\numberwithin{prop}{section}
\numberwithin{ques}{section}
\numberwithin{equation}{section}

\DeclareMathOperator{\HNN}{HNN}

\DeclareMathOperator{\rst}{res}

\newcommand{\ca}[1]{\mathcal{#1}}

\newcommand{\F}{\mathbb{F}}

\newcommand{\Z}{\mathbb{Z}}

\newcommand{\caA}{\ca{A}}

\newcommand{\tG}{\tilde{G}}

\newcommand{\argu}{\hbox to 7truept{\hrulefill}}
\newcommand{\G}{\ca{G}}

\def\Fin#1{{\it Fin(#1)}}
\begin{document}
\title{Infinitely Generated virtually free pro-$p$ groups and $p$-adic representations}
\author{ P. A. Zalesskii}\footnote{partially supported by CNPq}
\date{\today}

\address{P. A. Zalesski\u i\\
Department of Mathematics\\
University of Brasilia\\
70.910 Brasilia DF\\
Brazil}
\email{pz@mat.unb.br}

\begin{abstract}

We prove the pro-$p$ version of the Karras, Pietrowski, Solitar, Cohen and Scott result stating that a virtually free group acts on a tree with finite vertex stabilizers.  If a virtually free pro-$p$ group $G$ has finite centralizers of all non-trivial torsion elements, a stronger statement is proved: $G$ embeds into a free pro-$p$ product of a free pro-$p$ group and a finite $p$-group. Integral $p$-adic representation theory is used in the proof; it replaces  the Stallings theory of ends in the pro-$p$ case.

\end{abstract}

\maketitle



\section{Introduction}

Let $p$ be a prime number, and let $G$ be a pro-$p$ group containing
an open free pro-$p$ subgroup $F$, i.e. $G$ is  a virtually free pro-$p$ group. If $G$ is torsion free, then, according to the
celebrated theorem of Serre  \cite{S 65}, $G$ itself is free pro-$p$. This motivated him to ask the question whether the same statement holds
also in the discrete context. His question was answered positively some years later.
In several papers (cf. \cite{stall:bull}, \cite{stall:ends}, \cite{swan:coh1}),
J.R.~Stallings and R.G.~ Swan showed that free groups are precisely the groups of cohomological dimension $1$, and at the same time
J-P.~Serre himself showed that in a torsion free group $G$ the cohomological dimension of a
subgroup of finite index coincides with the cohomological dimension of $G$ (cf. \cite{ser:cohdis}).

One of the major tools for obtaining this type of results in the presence of torsion is the Stallings theory of ends.  Using it the description of a
virtually free discrete groups  as a group acting on a tree with finite vertex stabilizers was obtained  by Karras, Pietrowski and Solitar  \cite{KPS} in finitely generated case,  D. E. Cohen \cite{C}  when the group is countable and by G. P. Scott \cite{Scott} in the general case.

The objective of this paper is to prove a pro-$p$ version of this result that generalizes the main result of \cite{ihes}.

\begin{thm}\label{main} Let $G$ be a virtually free pro-$p$ group. Then $G$ acts on a pro-$p$ tree with finite vertex stabilizers.
\end{thm}

The theory of ends for pro-$p$ groups has been initiated in \cite{korenev}. However, it is not known whether
an analogue of Stallings' Splitting Theorem holds in this context. Here we use the representation theory over $p$-adic numbers; it is quite amazing that in the pro-$p$ situation it can replace the theory of ends.

More precisely,
first we embed $G$ into a semidirect product of $\tilde G=\tilde F\rtimes H$ of a free pro-$p$ group $\tilde F$ and a finite $p$-group $H$ such that all torsion elements of $\tilde G$ are conjugate into $H$.  Now we use
  integral $p$-adic representations of pro-$p$ groups. Namely, we prove that the abelianization $\tilde F^{ab}$ is a permutation $\Z_p[H]$-module. Then we use the  results of \cite{MSZ} on infinitely generated $p$-adic permutation modules and especially infinitely generated pro-$p$ version of the celebrated theorem of A. Weiss \cite{weiss} to prove the following theorem of independent interest.

 \begin{thm}\label{HNN} Let $G=F\rtimes H$ be a semidirect product of a finite $p$-group and a free pro-$p$ group $F$. Suppose the abelianization $F^{ab}$ is a $\Z_pH$-permutation module. Then $G$ is an HNN-extension with the base group $H$.\end{thm}

Thus $\tilde G$ is such an HNN-extension and  as in the classical Bass-Serre theory of groups acting on trees, such an HNN-extension acts on a standard pro-$p$ tree $T$ such that vertex stabilizers  are conjugate to  subgroups of $H$;  hence $G$ acts on $T$ as well. Moreover, this gives the following result.

\begin{thm}\label{embedding}  Let $F\rtimes H$ be a  semidirect product of a free pro-$p$ group $F$   and   a
finite $p$-group $H$. Then the action of $H$ extends to the action on some free pro-$p$ group  $\tilde F$ containing $F$  such that
$H$ permutes the elements of some basis of $\tilde F$.\end{thm}

The proof of Theorem \ref{main} also shortens and simplifies the proof of \cite[Theorem 1.3] {ihes}; it brings the complicated induction in $ G$ downstairs to permutation $\Z_p[H]$-module $M=F^{ab}$, where it is very transparent. In fact our proof practically does not use the theory of pro-$p$ groups acting on pro-$p$ trees.

  If a virtually free pro-$p$ group $G$ have finite centralizers of torsion elements then we show a much more precise result.

  \begin{thm}\label{finite centralizers embedding} Let $G$ be a virtually free pro-$p$ group having finite centralizers of the non-trivial torsion elements. Then $G$ embeds into a free pro-$p$ product $G_0=F_0\amalg H$ of a finite $p$-group $H$ and a free pro-$p$ group $F_0$.\end{thm}

  This theorem  can be considered as a generalization of the main result of \cite{McZ} where under the hypothesis of the second countability of $G$ was shown that $G$ is a free pro-$p$ product of finite $p$-groups and a free pro-$p$ group. Indeed, this result follows directly from Theorem \ref{finite centralizers embedding}  using the pro-$p$ version of the Kurosh Subgroup Theorem (\cite[Thm. 4.3]{Me} or \cite[Thm 9.6.2]{R 2017}. In general such a decomposition  might not exist. Considering however the natural action of $G$ on $G_0/H$  and denoting by $X$ the $G$-subset  of $G_0/H$ whose points  have non-trivial stabilizers,   we prove the following criterion of the existence of a decomposition of $G$ as  a free pro-$p$ product of finite $p$-groups and a free pro-$p$ group.

  \begin{thm}\label{decomposition} Let $G$ be a virtually free pro-$p$ group having finite centralizers of the non-trivial torsion elements. Then $G$ is a free pro-$p$ product of finite $p$-groups and a free pro-$p$ group if and only if  the natural quotient map $X\longrightarrow X/G$ by the action of $G$ admits a continuous section.

 \end{thm}

Such continuous section does not always exist according to \cite[Thm 10.7.4]{R 2017} or \cite[Section 4]{jalg}.

\subsection*{Acknowledgement}

The author  thanks  the anonymous referee for careful reading the manusrcript and many suggestions that lead to considerable improvement of the paper.

\medskip

{\bf Notation}

\medskip

If $G$ is a group and $H$ a subgroup of $G$ then $H^G$ denotes the normal closure of $H$ in $G$. By $G^{ab}$ we shall denote the abelianization of $G$. $\Z_p[G]$ shall denote the group ring if $G$ is finite and $\Z_p[[G]]$ the completed group ring if $G$ is infinite pro-$p$. If $M$ is a
$\Z_p[G]$-module, $M^G$ denotes invariants (i.e. fixed points) and $M_G$ coinvariants (meaning $M_G=M/I_G M$, where $I_G$ is the augmentation ideal). By ${\rm Res}^G_H(M)$ we shall denote the restriction of $M$ to $\Z_p[H]$.

\section{Preliminaries}
\label{s-HNN}

\subsection{Permutation modules}

\begin{defi} A boolean or profinite space $X$  is an inverse limit of
finite discrete spaces, i.e., a compact, Hausdorff, totally disconnected topological space. A pointed profinite space $(X,*)$ is a profinite space with a distinguished point $*$.  A profinite space
$X$ (resp. pointed profinite space $(X,*)$) with a profinite group $G$ acting continuously on it will
be called a $G$-space.\end{defi}

An example of a pointed profinite space is a one point compactification $X\cup \{*\}$ of a discrete space $X$ with added point $*$ being the distinguished point.
If $R$ is a commutative pro-$p$ ring with unity and $(X,*)=\varprojlim_{i\in I} (X_i,*)$ is an inverse limit of finite pointed spaces then $R[[X,*]]=\varprojlim_{i\in I} R[X_i,*]$ is a free pro-$p$ $R$-module  on the pointed space $(X,*)$.  Note that if $(X,*)$ is a one point compactification of a discrete space $X$ then   $R[[X,*]]=\prod_{|X|} R$. Remark also that if $*$ is an isolated point, then $R[[X,*]]=R[[X\setminus \{*\}]]$, i.e. $*$ becomes superfluous. If $G$ is a pro-$p$ group then $R[[G]]$ is a pro-$p$ ring, called a pro-$p$ group ring of $G$.

\begin{defi} Let $H$ be a pro-$p$-group and  $(X,*)$  a pointed $H$-space. Then the free abelian group  $R[[X,*]]$  becomes naturally a  pro-$p$ $R[[H]]$-module called {\it permutation} module. \end{defi}

The next proposition collects  facts on permutation pro-$p$ modules needed for this paper.

\begin{prop} \label{perm modules} Let $R=\Z_p$ or $\F_p$. Let $H$ be a finite $p$-group and $M$ be  a permutation pro-$p$ $R[H]$-module.

\begin{enumerate}

\item[(i)] $($\cite[Corollary 2.3]{McZ}$)$  $M=\bigoplus_{K\leq H}\prod_{I_K} R[H/K]$, where $I_K$ is some set of indices.

\item[(ii)] $($\cite[Prop 3.4]{MSZ}$)$ Every direct summand $A$ of $M$ is a permutation module. Furthermore  there exist a subset $J_K\subseteq I_K$  for each $K$ such that $\prod_{K\leq H}\prod_{J_K} R[H/K]$ complements $A$ in $M$.

\item[(iii)] $($\cite[Cor 6.8]{MSZ}$)$ If $U$ is also permutation $\Z_p[H]$-module then an extension of $M$ by $U$ splits and so is a permutation $\Z_p[H]$-module.

\end{enumerate}
\end{prop}

\begin{rem}\label{basis} $\prod_{I_K} \Z_p[H]$ is a free $\Z_p[H]$-module on the pointed profinite basis $(I_K,*)$, where $(I_K,*)$ is a one point compactification of a discrete space $I_K$. Then the composite $$( I_K,*)\longrightarrow \prod_{I_K} \Z_p[H]\longrightarrow \prod_{I_K} \Z_p[H/K]$$ is injective and so $( I_K,*)$ can be identified with its homeomorphic image  in $\prod_{I_K} \Z_p[H/K]$. So we can regard $X=(\bigcup_{K\leq H} I_K,*)$ as a subspace of $M=\bigoplus_{K\leq H}\prod_{I_K} \Z_p[H/K]$ that we shall call a {\it pointed basis} of $M$. Besides,
$\prod_{I_K} \Z_p[H/K]=\Z_p[H]\hat\otimes_{\Z_p[K]}\Z_p[[I_K,*]]$, where $\Z_p[[I_K,*]]$ is a trivial $\Z_p[[K]]$-module and
$\hat\otimes$ means the completed tensor product.\end{rem}

Recall that a $\Z_p[[G]]$-module $M$ is called a lattice if it is free as $\Z_p$-module.  One of the most beautiful results of modular representation theory of the end of 20-th century is the theorem of Weiss \cite{weiss} which infinitely generated
 pro-$p$ version is the following recently proved

\begin{thm}\cite[Thm 8.8]{MSZ}
\label{thm:weiss}
Let $G$ be a finite $p$-group and
let $M$ be a left $\Z_p[G]$-lattice. Let $N$ be a normal subgroup of $G$ such that
\begin{itemize}
\item[(i)] $\rst^G_N(M)$ is a free  $\Z_p[N]$-module, and
\item[(ii)] $M^N$ is a $\Z_p[G/N]$-permutation module.
\end{itemize}
Then $M$ is a $\Z_p[G]$-permutation module.
\end{thm}

The proof of this theorem as well as of the original Weiss's theorem is quickly reduced (by induction on $N$) to the case
 $|N|=p$. In this case however we can weaken the hypothesis (i) of Theorem \ref{thm:weiss} assuming that $M$ is $\Z_p[N]$-permutation and that a maximal  $\Z_p[N]$-trivial summand of it is $G$-invariant.

 \begin{thm}
\label{generalized weiss}
Let $G$ be a finite $p$-group and
let $M$ be a left $\Z_p[G]$-lattice. Let $N$ be a normal subgroup of $G$ of order $p$ such that
\begin{itemize}
\item[(i)] $\rst^G_N(M)$ is a permutation  $\Z_p[N]$-module $M_1\oplus M_p$ with $M_p$ being $\Z_p[N]$-free and  $M_1$ being $G$-invariant and $\Z_p[N]$-trivial;
\item[(ii)] $M^N$ is a $\Z_p[G/N]$-permutation module.
\end{itemize}
Then $M$ is a $\Z_p[G]$-permutation module and $M_1$ is its direct summand.
\end{thm}

\begin{proof}  We have a $\Z_p$-decomposition into direct sum $M^N=M_1\oplus M_p^N$  and so  modulo $p$ we have an $\F_p$-decomposition $\bar M^N =\bar M_1\oplus \bar M_p^N$ of $\F_p$-modules. Now applying an operator $\sum_{n\in N} n$ to $\F_p[N]$-module $\bar M=\bar M_1\oplus \bar M_p$  we obtain $\bar M_p^N$ and so  it is an $\F_p[G]$-submodule of $\bar M$. Hence
$\bar M^N=\bar M_1\oplus \bar M_p^N$ is an $\F_p[G]$-decomposition and so by Proposition
\ref{perm modules}(ii)  $\bar M_1$ and $\bar M_p^N$ are $\F_p[G]$-permutation, since $\bar M^N$ is. By  \cite[Theorem 8.6]{MSZ} $M_1$ and $M_p^N\cong M^N/M_1$ are  $\Z_p[G]$-monomial lattices, (i.e. a direct product of  lattices  induced from a rank 1 lattices).   Monomial $\Z_p[G]$-lattices are permutation $\Z_p[G]$-lattices for $p>2$. For $p=2$ we need to apply \cite[Lemma 8.4]{MSZ} to deduce that  the $\Z_p[G]$-module $M^N$ decomposes as  $M^N\cong M_1\oplus M_p^N$ and hence by Proposition \ref{permutation module} (ii) $M_1$ and $M_p^N\cong M^N/M_1$ are $\Z_p[G]$-permutation.  Applying Theorem \ref{thm:weiss} we deduce that  $M_p$ is a $\Z_p[G]$-permutation module.  Therefore $M$ is an extension of  permutation $\Z_p[G]$-modules ($M_1$ by $M_p$) and so by Proposition \ref{perm modules} (iii) is a permutation $\Z_p[G]$-module with $M_1$ being  its direct summand.

\end{proof}

\begin{rem}\label{converse}  Theorem \ref{generalized weiss} is the main ingredient of induction in the proof of the main result.  We note that hypotheses (i) and (ii) of Theorem \ref{generalized weiss} are also necessary. Indeed,   if $M$ is $\Z_p[G]$-permutation then (i) is clearly satisfied. To see (ii) take a typical direct summand $\Z_p[G/H]$ of $M$, where $H$ is a subgroup of $G$. Observe that if $N\leq H$ then $\Z_p[G/H]^N=\Z_p[G/H]$. Otherwise $\Z_p[G/H]$ is $\Z_p[N]$-free and so $\Z_p[G/H]^N=\Z_p[G/H]_N=\Z_p[G/HN]$ (see \cite[Lemma 2.4]{McZ}).\end{rem}

The condition of $G$-invariancy is essential as the following example shows.

\begin{example} Let $G=N\times C$ be a direct product of two groups of order $2$ with generators $c_1$ and $c_2$ respectively. Let $M=\langle x,a,b\rangle$ be a free abelian pro-$2$ group of rank 3. Define an action of $G$ on $M$ as follows: $c_1x=x, c_1a=b, c_1b=a$, $c_2x=x+a+b$, $c_2a=-a, c_2b=-b$. Then  ${\rm Res}^G_N(M)=\Z_2\oplus \Z_2[C_2]$ and so is permutation. The $\Z_2[G/N]$-module $M^N\cong \Z_2[G/N]$ and so is permutation. But $M$ is not permutation $\Z_2[G]$-module, since it is not permutation as $\Z_2[C]$-module.

\end{example}

\subsection{HNN-extensions}

If $(X,*)$ is a pointed profinite space, a free pro-$p$ group $F=F(X,*)$ is a pro-$p$ group together with a continuous map $\omega: (X,*)\longrightarrow (F,1)$ satisfying the following universal property:

$$
\xymatrix{F(X,*)\ar@{-->}[rd]^{\eta}& \\
            (X,*)\ar[u]^{\omega}\ar[r]^{\alpha}&G}$$

for any pro-$p$ group $G$, any  continuous map of pointed profinite spaces $\alpha: (X,*)\longrightarrow (G,1)$
extends uniquely to a continuous homomorphism $\eta:F(X,*)\longrightarrow G$.

If $(X,*)$ is a one point compactification of a discrete space  $X$ then the free pro-$p$ group $F(X,*)$ on the pointed profinite space $(X,*)$ coincides with the notion of a free pro-$p$ group $F(X)$  on the basis $X$ convergent to $1$ (see \cite[Chapter 3]{RZ}).

\medskip

We introduce a notion of a {\em pro-$p$ HNN-extension} slightly simpler than  the construction described in \cite{HZ 07} to make the paper more accessible to the reader; this version suffices for our purpose. In the literature usually the term HNN-extension stands for the construction with one stable letter only, in our case through the paper we shall have a set of stable letters.

\begin{defi}\label{HNN-ext} Suppose that $G$ is a pro-$p$ group, and for a finite set
$\{A_i\mid i\in I\}$ of subgroups of $G$ and a family of pointed profinite spaces $\{(X_i,*)\mid i\in I\}$ there are given continuous maps  $\phi_i:A_i\times X_i\to G$ such that $(\phi_i)_{|A_i\times \{x\}}$ is an injective homomorphism and $(\phi_i)_{|A_i\times \{*\}}=id$ for each $x\in X_i$. The {\em HNN-extension} $\tG:=\HNN(G,A_i,\phi_i, X_i)$, $i\in I$
is defined to be  the quotient of a free pro-$p$ product $ \coprod_{i\in I} F(X_i,*)\amalg G$ modulo the
relations $\phi_i(a_i,x_i)=x_ia_ix_i^{-1}$ for all $x_i\in X_i$, $i\in I$.
We call
$G$ the {\em base group}, $A_i$ and $B_{ix}:=\phi_i(A_i\times\{x\} )$ {\em associated} subgroups and $\bigcup_{i\in I} X_i\setminus \{*\}$ the set of {\em stable letters}. Note that if $A_i$ and $A_j$ are conjugate then the relation $\phi_i(a_i)=x_ia_ix_i^{-1}$ implies automatically the corresponding relation on $A_j$, therefore we can remove $j$ from $i$ enlarging $X_i$; thus w.l.o.g we shall assume from now on that  $\{A_i\mid i\in I\}$ is the family of pairwise non-conjugate subgroups.

If each $(\phi_i)_{|A_i\times \{x\}}$ is an identity map, we call the HNN-extension {\it special}.
\end{defi}

\begin{rem}\label{action on tree}
A pro-$p$ HNN-group $\tilde G$ is a special case of the fundamental pro-$p$ group $\Pi_1(\G, \Gamma)$
of a profinite graph of pro-$p$ groups $(\G, \Gamma)$  (see  \cite[Example (e)]{R 2017}). Namely, a pro-$p$ HNN-group can be thought as $\Pi_1(\G, \Gamma)$, where $\Gamma$ is a bouquet (i.e., a connected
profinite graph having just one vertex - the distinguished  point of the one point compactification $\Gamma=(\bigcup_{i\in I} X_i,*)$ of the set of stable letters). In particular, a pro-$p$ HNN-group $\tG:=\HNN(G,A_i,\phi_i, X_i)$ acts on a pro-$p$ tree $S=S(\tilde G)$ defined as follows:
the vertex set $V(S)=\tilde G/G$ and the edge set $E(S)=\bigcup_{i\in I} \tilde G/A_i\times (X_i\setminus \{*\})$ with $\tilde gG$ and $\tilde g x_i G$ to be initial and terminal vertices of $(\tilde g A_i,x_i)$, $x_i\in X_i$. The fact that $S(\tilde G)$ is a pro-$p$ tree means that there is an exact sequence of permutation $\Z_p[[G]]$-modules

$$0\rightarrow \Z_p[[E(S),*]]\rightarrow \Z_p[[V(S)]]\rightarrow \Z_p\rightarrow 0$$
where the second arrow is given by $(\tilde g A_i,x_i)\rightarrow \tilde g x_i G - \tilde g G$.\end{rem}

The next proposition  corrects unessential error in the statement of Lemma 3.8 in \cite{ihes}, namely the description of the centralizers as $C_{\tilde F}(A_i)=\coprod_{s\in S_i}F(Z_i,*)^s$ there was correct only for $A_i$ maximal finite (by inclusion), and only for maximal $A_i$ used.  One can consult \cite{HZ} for details. We use the left action by conjugation by inverses since it corresponds better to the left modules obtained after the abelianization. 

 \begin{prop}\label{perm-ext0} Let $\tilde G=\HNN(K,A_i,X_i,i\in I)$ be a special pro-$p$ HNN-extension  of a finite $p$-group $K$ and $\tilde F$ be the normal closure of $F(\bigcup_{i\in
I}X_i,*)$ in $\tilde G$. For every $i\in I$
choose respectively coset representative sets $R_i$ of
$K/N_K(A_i)$ and $S_i$ of $N_K(A_i)/A_i$. Then $$\tilde F =\coprod_{i\in I}\coprod_{r\in R_i}\coprod_{s\in S_i}F(X_i,*)^{{s^{-1}r^{-1}}}.$$

\end{prop}

\begin{proof} By Definition \ref{HNN-ext}, one can view  $\tilde G$ as the quotient of
 $G:=F(\bigcup_{i\in I}X_i)\amalg K$ modulo the relations
$[a_i,x_i]$ for all $x_i\in X_i$ and $a_i\in A_i$, with $i$
running through the finite set $I$. By the  Kurosh Subgroup Theorem  (see \cite[Theorem 9.1.9]{RZ}) applied to the normal closure $N$ of
$F(\bigcup_{i\in I}X_i,*)$ in $G$ we have a free pro-$p$ decomposition
$$N=\coprod_{i\in I}\coprod_{r\in R_i}\coprod_{s\in
S_i}\coprod_{a\in A_i}F(X_i,*)^{a^{-1}s^{-1}r^{-1}}.$$
The relations yield
$F(X_i,*)^{a^{-1}}=F(X_i^{a^{-1}},*)=F(X_i,*)$.
Since for $s\in S_i, a\in A_i, x\in X_i$ one has
$[a,x]=1$ if, and only if, $[a^s,x]=1$ if and only if $[a,x^{s^{-1}}]=1$, we deduce that  $$\tilde F=\coprod_{i\in I}\coprod_{r\in R_i}\coprod_{s\in
S_i}F(X_i,*)^{s^{-1}r^{-1}}.$$

 \end{proof}

\medskip
 Observing that $K$ acts on $\tilde F =\coprod_{i\in I}\coprod_{r\in R_i}\coprod_{s\in S_i}F(X_i,*)^{s^{-1}r^{-1}}$ permuting the free factors we deduce the following

\begin{cor} \label{abelianization of special HNN} $$\tilde F^{ab} =\bigoplus_{i\in I}\prod_{z\in X_i} \Z_p[K/A_i],$$
with $(X_i, *)[\tilde F,\tilde F]/[\tilde F,\tilde F]$ being a pointed basis of $\prod_{x\in X_i} \Z_p[K/A_i]$. Moreover,  $X_i\subset C_{\tilde  F}(A_i)[\tilde F,\tilde F]/[\tilde F,\tilde F]$.

\end{cor}

\section{The structure of the abelianization}

We start this section with the lemma proved for finitely generated case in \cite{HZ 08} that relates a semidirect product $G=F\rtimes C_p$ of a free pro-$p$ group of order $p$ and a group $C_p$ of order $p$ with $\Z_p[C_p]$ module $F^{ab}$. The proof however uses only finiteness of the set of conjugacy classes of groups of order $p$ in $G$. We give the proof here as well for convenience of the reader.

\begin{lem}\label{free-by-Cp} Let $G=F\rtimes C$ be a semidirect product of a finite cyclic $p$-group $C=\langle c\rangle$ and free pro-$p$ group $F$. Suppose $G$ has only one conjugacy class of groups of order $p$. Then

\begin{enumerate}
\item[(i)]  $($\cite[Theorem 1.2]{crelle}, \cite[Prop. 10.6.2]{R 2017}$)$ $G$ has a free pro-$p$ product  decomposition
 $G=\left( C\times H_0\right)\amalg H$, such that $H_0$ and $H$ are free pro-$p$ groups
contained in $F$.

\item[(ii)] The $\Z_p[C]$-module $M=F^{ab}$ decomposes in the form
$$M=M_1\oplus M_p$$
such that $M_1$ is trivial and $M_p$ is a free $\Z_p[C]$-module.

\end{enumerate}

Moreover,  if $X_0$ is the basis of $H_0$ and $X$ is the basis of $H$ then the basis of $M_1$ is $ X_0[F,F]/[F,F]$ and the $\Z_p[C]$-basis of $M_p$ is $X[F,F]/[F,F]$.
\end{lem}

\begin{proof} (ii) is a particular case of Corollary \ref{abelianization of special HNN} but the proof in this case is much easier so we give it here. Denote by $c$ a generator of $C$.
By the pro-$p$ version of the Kurosh subgroup Theorem,
\cite[Thm. D.3.8]{RZ}
we have
$$F=H_0\amalg \left(\coprod_{j=0}^{p-1}H^{c^j}\right).$$

Factoring out $[F,F]$ one arrives at the desired decomposition.

\end{proof}

We shall need the following proposition, where $tor(G)$ means the set of non-trivial torsion elements.

\begin{prop}\cite[Prop 1.7 (i)]{crelle}\label{lifting torsion} Let $G$ be a virtually free pro-$p$ group and $N$ a normal subgroup of $G$ generated by torsion. Then

$$tor(G)N/N=tor(G/N),$$
where $tor(G)$ stands for the torsion of $G$.
\end{prop}

\begin{lem}\label{quotient}  Let $G=F\rtimes H$ be a semidirect product of a finite $p$-group $H$ and a free pro-$p$ group $F$. Suppose every torsion element of $G$ is conjugate into $H$.
 If $C$ is central subgroup of order $p$ in $H$, then
$G/C^G=\bar F\rtimes H/C$ with $\bar F$ free pro-$p$ and every finite subgroup of $G/C^G$ is conjugate to $H/C$.
\end{lem}

\begin{proof}

 By Proposition  \ref{lifting torsion} every torsion element $k$ of $G/C^G$ is the image of some torsion element $g$ of $G$. By hypothesis $g$ is conjugate into $H$ so $k$ is conjugate into  $H/C$. By Lemma \ref{free-by-Cp}(i) $\bar F=FC^G/C^G$ is free pro-$p$.\end{proof}

\begin{thm}\label{permutation module} Let $G=F\rtimes H$ be a semidirect product of a finite $p$-group $H$ and a free pro-$p$ group $F$. Suppose every torsion element of $G$ is conjugate into $H$. Then $F^{ab}$  is a $\Z_p[H]$-permutation module. \end{thm}

\begin{proof} We use induction on the order of $H$.  Let $C$ be a central subgroup of $H$ of order $p$.  By Lemma \ref{free-by-Cp} $F^{ab}$ is $\Z_p[C]$-permutation with  $F^{ab}=M_1\oplus M_p$, where $M_1$ is trivial and $M_p$ is a free $\Z_p[C]$-module and moreover,  $M_1$ is the image of $C_F(C)$ in $F^{ab}$.  Since $C_F(C)$ is $H$-invariant so is $M_1$.

We want to apply  Theorem \ref{generalized weiss} to show that $F^{ab}$ is a $\Z_p[H]$-permutation module and we showed in the previous paragraph that the hypothesis (i) of Theorem \ref{generalized weiss} is satisfied.    Consider the commutative diagram

$$\xymatrix{F\ar[r]\ar[d]& FC/C^G\ar[d]\\
                F^{ab}\ar[r]& F^{ab}_C},$$
                where the bottom row is just the abelianization of the upper one.
We can apply Lemma \ref{quotient}  and the induction hypothesis to $G/C^G$  to deduce that $F^{ab}_C$ is an $\Z_p[H/C]$-permutation module. Now $F^{ab}_C/M_1\cong  (M_p)_C$ and by Lemma \cite[Lemma 2.4]{McZ}  $(M_p)_C\cong M_p^C$, so $(F^{ab})^C\cong F^{ab}_C$ is a permutation $\Z_p[H/C]$-module.

Thus the hypothesis (ii) of Theorem \ref{generalized weiss} is also satisfied and so  $F^{ab}$ is a permutation $\Z_p[H]$-module.\end{proof}

\section{Permutational abelianization}

\begin{prop}\label{free-by-p} Let $G=F\rtimes C$ be a semidirect product of a group $C=\langle c\rangle$ of order $p$ and a free pro-$p$ group $F$. Then $F^{ab}$ is $\Z_p[C]$-permutation if and only if $C$ is the unique subgroup of $G$ of order $p$ up to conjugation.
 \end{prop}

\begin{proof} `if' is a particular case of Theorem \ref{permutation module}.

`Only if'.  By \cite[Proposition 1.3]{crelle} (or by Lemma \ref{mono}) combined with Lemma \ref{free-by-Cp} $G$ embeds into a pro-$p$ group $G_0=(C_p\times H)\amalg H_0$, where $C_p$ is of order $p$ and $H,H_0$ are free pro-$p$. Thus w.l.o.g we may assume that $G\leq G_0$ and $C=C_p$.  By the pro-$p$ version of the Kurosh subgroup theorem (see \cite[Theorem 9.1.9]{RZ})  any open subgroup $U$ of $G_0$ containing $G$ admits  a decomposition $U=\coprod_{g\in S_U} (U\cap (C\times H)^g) \amalg H_U$ for some free pro-$p$ group $H_U$, where $S_U$ is the set of representatives of $U\backslash G_0/(C\times H)$.  Let $T$ be a subgroup of $G$ order $p$ which is not conjugate to $C$. Then we can choose $U$ such that $C$ and $T$ are not conjugate in $U$ and in fact assume w.l.o.g $T=C^g$ for some $g\in S_U$. Thus we can regard $C\amalg T$ as a subgroup of $G$ and $U$. Put $\bar F=F\cap (C\amalg T)$ and note that it is finitely generated. Then the natural splitting epimorphism $U\longrightarrow C\amalg T=\bar F\rtimes C$ restricted to $F$ gives the following commutative diagram

$$\xymatrix{\bar F\ar[r]\ar[d] &F\ar[r]\ar[d]& \bar F\ar[d]\\
            \bar F^{ab}\ar[r]&F^{ab}\ar[r]& \bar F^{ab}}$$ with upper and lower composition maps being identity.
               So $\bar F^{ab}$ is a finitely generated direct summand of $F^{ab}$ and by \cite[Lemma 6]{HZ 08} $\bar F^{ab}$ is not a permutation $\Z_p[C]$-module. By Proposition \ref{perm modules}(ii) this contradicts  to the hypothesis that $F^{ab}$ is $\Z_p[C]$-permutation.

\end{proof}

From now on  we shall assume in this section that $G=F\rtimes H$ with $H\neq \{1\}$ and $F^{ab}$  being a $\Z_p[H]$-permutation module, i.e. $F^{ab}=\bigoplus_{e\in E} M_e$, where $M_e= \prod_{x\in X_e} \Z_p[H/H_e]=\Z_p[H/H_e]\hat \otimes \Z_p[[X_{e},*]]$ and $\{H_e, e\in E\}$ is a set of representatives of
conjugacy classes of certain subgroups of $H$ and $(X_{e},*)$ is the one point compactification of $X_{e}$. Thus we view $X_e$ as a subset of $M_e$ and in fact $(X_e,*)$ is a pointed basis of $M_e$ as explained in Remark \ref{basis}.

We shall relate centralizers of subgroups of $H$ in $F$ with the $\Z_p[H]$-module structure of $F^{ab}$.

\begin{lem} \label{H=He} If there exists $e_0\in E$ such that $H_{e_0}=H$ for some $e_0\in E$, then the decomposition $F^{ab}=\bigoplus_{e\in E} M_e$ can be rearranged such that   $C_F(H)[F,F]/[F,F]= M_{e_0}=\Z_p[[X_{e_0},*]]$ is a direct $\Z_p[H]$-trivial summand of $F^{ab}$. \end{lem}

\begin{proof}  We argue by induction on $H$ using Proposition \ref{free-by-p} as the base of induction. Let $C$ be a central subgroup of order $p$ in $H$.  Then ${\rm Res}_C^H(F^{ab})$ is a permutation $\Z_p[C]$-module and so by Proposition \ref{free-by-p} $FC$ has only one conjugacy class of groups of order $p$. Therefore by Lemma \ref{free-by-Cp} (ii)   $F^{ab}=M_1\oplus M_p$, where $$M_1=C_F(C)^{ab}=C_F(C)[F,F]/[F,F]$$ is a trivial $\Z_p[C]$-module     and $M_p$ is a free $\Z_p[C]$-module generated by $X_C[F,F]/[F,F]$.  Moreover, $M_1$ is $\Z_p[H]$-submodule since $C_F(C)$ is $H$-invariant. Note also that $M^C$ is $\Z_p[H/C]$-permutation since $M$ is $\Z_p[H]$-permutation (see Remark \ref{converse}). Hence by Theorem \ref{generalized weiss} $M_1$ is a direct summand of $F^{ab}$. Then taking into account that $\Z_p[H/H_e]$ is trivial $\Z_p[C]$-module if and only if $C\leq H_e$ one deduces from  Proposition \ref{perm modules} (ii) that  the decomposition  $F^{ab}=\bigoplus_{e\in E} M_e$ can be rearranged such that $M_1=\bigoplus_{C\leq H_e} M_e$. Now $C_F(H)\leq C_F(C)$, $HC_F(C)/C=C_F(C)\rtimes H/C$  and $C_F(C)^{ab}=M_1$, so  by induction hypothesis  $C_F(H)[F,F]/[F,F]=M_{e_0}$ as required.

\end{proof}

\begin{prop} \label{abelianization general centralizers} Let $(X_e,*)$ be a pointed profinite basis of $M_e$, $e\in E$. Then   
the decomposition $\bigoplus_{e\in E}M_e$ can be rearranged such that  
$$X_e\subset C_F(H_{e})[F,F]/[F,F]. \eqno(1)$$ for all $e\in E$.

 \end{prop}

\begin{proof}  We use induction on $|H|$.

If $H=H_{e_*}$ for some $e_*\in E$ then  by Lemma \ref{H=He}   $F^{ab}=\bigoplus_{e\in E}M_e$ can be rearranged such that   $C_F(H)[F,F]/[F,F]=M_{e_*}$,  is a $\Z_p[H]$-trivial direct summand of $F^{ab}$. So we are done if $E=\{e_*\}$.

         Let $K$ be a subgroup of $H$ of index $p$.  If $H_e\leq K$, then   ${\rm Res}^H_K(M_e)=\prod_{x\in X_e} \bigoplus_{r\in R}  \Z_p[[K/H_e^{r}]]$, where $R$ is a set of  coset representatives of $H/K$. Otherwise, ${\rm Res}^H_K(M_e)=\prod_{x\in X_e}   \Z_p[[K/H_e\cap K]]$.  So putting $E_K=\{e\in E\mid H_e\leq K\}$ one has $${\rm Res}^H_K(F^{ab})=\bigoplus_{e\in E_K} \prod_{x\in X_e} \bigoplus_{r\in R}  \Z_p[[K/H_e^{r}]]\oplus\bigoplus_{e\in E\setminus E_K}\prod_{x\in X_e}  \Z_p[[K/H_e\cap K]]. $$

    By induction hypothesis (applied to $FK$) this decomposition can be rearranged such that    $X_e\subset C_F(H_{e})[F,F]/[F,F]$ for every $e\in E_K$ (in fact $rX_e\subset C_F(H_{e}^r)[F,F]/[F,F]$ for every $r\in R$ and $e \in E$, but we need it only for $r=1$). Since $\bigoplus_{r\in R}\prod_{x\in X_e}  \Z_p[[K/H_e^{r}]]=M_e$  the  statement is proved for every $e\in E_K$. 
    
    Let $E_0$ be a maximal subset of  $ E$ such that there exists a decomposition of $F^{ab}$ with (1)  satisfied for all $e\in E_0$. Since for a given $e\neq e_*$ there exists  $K$ of $H$ of index $p$ containing $H_e$,  the preciding paragraph shows  that  if $E\neq \{e_*\}$ then $E_0$ is not empty.

     Pick $e'\in E\setminus E_0$. Let $K$ be a subgroup of index $p$ in $H$ containing $H_{e'}$. Then by above there is a maximal direct summand $M'_{e'}$ of $F^{ab}$ isomorphic to $M_{e'}$ satisfying the statement of the proposition. Then by Proposition \ref{perm modules} (ii) $F^{ab}=\bigoplus_{e\neq e'} M_{e}\oplus M'_{e'}$ contradicting the maximality of $E_0$. Thus $E_0=E$ and the proposition is proved.
\end{proof}

\section{Special HNN-extensions}
We are ready to prove Theorem \ref{HNN}, where
 we shall keep the notation $M_e= \prod_{x\in X_e}\Z_p[H/H_e]$ viewing $X_e$ as a subset of $M_e$ (cf. Remark \ref{basis}).

\begin{thm}\label{perm HNN} Let $G=F\rtimes H$ be a semidirect product of a finite $p$-group $H$ and a free pro-$p$ group $F$. Suppose $F^{ab}=\bigoplus_{e\in E} M_e$ is $\Z_p[H]$-permutation module. Then $G={\rm HNN}(H, H_e, X_e)$ is a special HNN-extension with the base group $H$.\end{thm}

\begin{proof}

\medskip 
By Proposition \ref{abelianization general centralizers}
 the decomposition $\bigoplus_{e\in E}M_e$ can be rearranged such that for all $e\in E$ the pointed profinite basis  $(X_e,*)$ of $M_e$ is contained in $ C_F(H_e)[F,F]/[F,F]$. Since the natural homomorphism $C_F(H_e)\longrightarrow  C_F(H_e)[F,F]/[F,F]$ admits a continuous section (see \cite[Lemma 5.6.5]{RZ})  the  injective continuous map $\eta^{ab}_e:(X_e,*)\longrightarrow M_e $; $x\longrightarrow 1H_e\hat\otimes x$ defined in Remark \ref{basis} lifts to a continuous map  $\eta_e:(X_e,*)\longrightarrow C_F(H_e)$.
Let $\tilde G={\rm HNN}(H, H_e, X_e)$ be a special HNN-extension and  $f:\tilde G\longrightarrow G$ be a homomorphism given by the universal property that extends $\eta_e$ and sends $H$ identically to $H$. Let $\tilde F=f^{-1}(F)$. Then by Corollary \ref{abelianization of special HNN} $\tilde F^{ab}=\bigoplus_{e\in E} M_e$,  such that   $X_e\subset C_{\tilde F}(H_{e_0})[\tilde F,\tilde F]/[\tilde F,\tilde F]$ for every $e\in E$. Thus one has  the commutative diagram

$$\xymatrix{\tilde F\ar[dr]\ar[rr]^f&&F\ar[dl]\\
&\bigoplus_{e\in E} M_e&}$$
implying that $f_{|\tilde F}$ induces an isomorphism on the abelianizations and therefore  $f$ has to be an isomorphism.\end{proof}

Combining this theorem with Theorem \ref{permutation module} we deduce the following

\begin{cor}\label{permext} Let $G=F\rtimes H$ be a semidirect product of finite $p$-group $H$ and free pro-$p$ group $F$. Suppose every torsion element of $G$ is $F$-conjugate into $H$. Then $G$ is a special HNN-extension with the base group $H$.\end{cor}

We are now ready to prove  Theorem \ref{embedding}

\begin{thm}\label{immersion}  Let $F\rtimes H$ be a  semidirect product of a free pro-$p$ group $F$   and   a
finite $p$-group $H$. Then the action of $H$ extends to the action on some free pro-$p$ group  $\tilde F$ containing $F$  such that
$H$ permutes the elements of some (pointed) basis of $\tilde F$.\end{thm}

\begin{proof}   Use an embedding $G=F\rtimes H \longrightarrow \tilde G=\tilde F\rtimes H$ (see Lemma \ref{mono}). Every torsion element of $\tilde G$ is $\tilde F$-conjugate into $H$ so  $\tilde G={\rm HNN}(H, H_e, X_e)$, $e\in E$
 is a special HNN-extension by Corollary \ref{permext}. By Lemma \ref{perm-ext0}

$$\tilde F =\coprod_{e\in E}\coprod_{r\in R_e}\coprod_{s\in S_e}F(X_e,*)^{s^{-1}r^{-1}},$$ is a free pro-$p$ product of free pro-$p$ groups $F(X_e,*)$ on the pointed basis $(X_e,*)$ where  $R_e$, $S_e$ are coset representative sets of
$H/N_H(H_e)$ and  $N_H(H_e)/H_e$ respectively. Since $H_e$ centralizes $X_e$,  the action of $H$ on $E$ by conjugation  permutes the free factors  $F(X_e,*)^{s^{-1}r^{-1}}$ with

 $$\bigcup_{e\in E}\bigcup_{r\in R_e}\bigcup_{s\in S_e}(X_e^{s^{-1}r^{-1}},*),$$
 being the $H$-invariant pointed basis of $\tilde F$.

\end{proof}

Now Theorem \ref{main} follows easily.

\begin{thm}\label{acting on trees} Let $G$ be a virtually free pro-$p$ group. Then $G$ acts on a pro-$p$ tree with finite vertex stabilizers.

\end{thm}

\begin{proof}
Let $F$ be an open normal subgroup of $G$.  We can embed $G$ into $G_0=G\amalg G/F=F_0\rtimes G/F$, where $F_0$ is the kernel of the natural epimorphism $G_0\longrightarrow G/F$ induced by the natural epimorphism $G\longrightarrow G/F$ and the identity map on the second factor. Note that by the pro-$p$ version of the Kurosh subgroup theorem \cite[Theorem 9.1.9]{RZ}$F_0$ is free pro-$p$. By Theorem \ref{immersion} $G_0$  can be embedded in a semidirect product
$\tilde G = \tilde F\rtimes G/F$ such that $G/F$ permutes the elements of some (pointed) basis of $\tilde F$ and so $\tilde F^{ab}$ is a $\Z_p[G/F]$-permutation module.    Hence by Theorem \ref{perm HNN} $\tilde G$ is a special HNN-extension with finite base group $G/F$.  Therefore according to Remark \ref{action on tree} $\tilde G$ and hence $G$ acts on a pro-$p$ tree with finite vertex stabilizers.
\end{proof}

We finish this section with  the result on virtually free pro-$p$ groups that resembles the statement of the  Weiss theorem for modules.

\begin{thm}\label{rigidity} Let $G=F\rtimes H$ be a semidirect product of finite $p$-group and free pro-$p$ group $H$. Let $N$ be a normal subgroup of $H$ such that

\medskip
(i) all torsion elements of $FN$ are conjugate into $N$;

(ii) all torsion elements of $G/N^G$ are conjugate into $HN^G/N^G$.

\medskip
Then $G$ is a special HNN-extension with base group $H$. In particular, all torsion elements of $G$ are conjugate into $H$. \end{thm}

\begin{proof} We shall use induction on $|N|$.
Suppose $|N|=p$. Then by Lemma \ref{free-by-Cp} $\Z_p[N]$-module $M=F^{ab}$ has a decomposition $M_1\oplus M_p$ such that $M_1$ is trivial and $M_p$ is a free $\Z_p[N]$-module and moreover $M_1$ is $\Z_p[H]$-invariant. Hence $M_N\cong M^N$. Then observing that $N^G=N^{FN}$ we have the following commutative diagram:

$$\xymatrix {F\ar[r]\ar[d]&\bar F=FN/N^{G}\ar[d]\\
              M\ar[r] & M_N},$$
where  by Lemma \ref{quotient} $\bar F$ is free pro-$p$.

Clearly the $\Z_p$-module $M_N=M_1\oplus (M_p)_N$.

By Theorem \ref{permutation module} $M^N\cong M_N=\bar F^{ab}$ is a permutation $\Z_p[H]$-module and so by Theorem \ref{generalized weiss}  so is $M$.   Hence by  Theorem \ref{perm HNN} $G$ is a special HNN-extension of $H$. In particular, by \cite[Cor. 7.1.3]{R 2017} all torsion elements of $G$  are conjugate into $H$ and in fact $F$-conjugate. 

Suppose now $|N|>p$ and
 let $C$ be a central normal subgroup of $H$ contained in $N$. By   Lemma \ref{quotient}  $G/C^G=\bar F\rtimes H/C$, where $\bar F=FC/C^G$ is free pro-$p$ and    every torsion element of $\bar F\rtimes  N/C$ is conjugate into $N/C$ (since  $C^G=C^{NF}$). Hence by the induction hypothesis  all torsion elements of $G/C^G$ are conjugate into $H/C$. Then by the base of induction all torsion elements of $G$ are conjugate into $H$ and by Corollary \ref{permext} $G$ is a special HNN-extension.

\end{proof}

\section{Finite centralizers of torsion}

The main theorem of \cite{McZ} states that a second countable virtually free pro-$p$ group having finite centralizers of the non-trivial torsion elements is a free pro-$p$ product of finite groups and a free pro-$p$ group. It is not true for virtually free pro-$p$ groups of uncountable rank. The objective of this section is to show that a free pro-p groups of arbitrary rank  embeds in a free pro-$p$ product of a finite $p$-group and a free pro-$p$ group. We shall also establish a criterion  for its decomposition as a free pro-$p$ product of finite $p$-groups and a free pro-$p$ group.

We shall begin with the following

\begin{lem} \label{centralizers out}  Let $\tilde G=\HNN(K,A_i,X_i,i\in I)$ be a special pro-$p$ HNN-extension  of a finite $p$-group $K$ and $\tilde F$ be the normal closure of $(\bigcup_{i\in
I}X_i,*)$ in $\tilde G$. Let $I_0=\{i\in I\mid A_i\neq 1\}$ and $F$ be a $K$-invariant free factor of $\tilde F$  such that  $C_{ F}(A_i)=1$   for every $i\in I_0$. Then putting $\tilde F_K=\langle X_i\mid i\in I_0\rangle^{\tilde F}$ we have $F\cap \tilde F_K =1$ and $\tilde G/\tilde F_K=F_0\amalg K$ for some free pro-$p$ group $F_0$.\end{lem}

\begin{proof} The second statement follows directly from the presentation of a special HNN-extension.
     We shall induct on $|I_0|+|K|$ to prove the first one.

      If $K$ is of order $p$ (and so $|I_0|\leq 1$) then by Lemma \ref{free-by-Cp} $\tilde G=(K\times F_K)\amalg F_0$, where $F_K=\langle X_i\mid i\in I_0\rangle$, $F_0$ are free and  $\tilde F^{ab}=M_1\oplus M_p$, where $M_1$ is the trivial $\Z_p[K]$-module coinciding with the image of $F_K$ in $\tilde F^{ab}$ and $M_p$ is a free $\Z_p[K]$-module. Thus we have the following commutative diagram
     $$\xymatrix{\tilde F\ar[r]\ar[d] & F_0\ar[d]\\
               \tilde F^{ab}\ar[r]& M_p}.$$

               Since $F$ is a $K$-invariant free factor of $\tilde F$, its abelianization
               $F^{ab}$ is a pure $\Z_p[K]$-submodule of $\tilde F^{ab}$ (i.e. a $\Z_p$-direct summand). By \cite[Theorem B]{porto} $F^{ab}=M_{p-1}\oplus L$, where $L$ is a free $\Z_p[K]$-module and $M_{p-1}$ has no non-zero elements fixed by $K$. Since $L$ is free it is a direct summand of $\tilde F^{ab}$  (see \cite[Proposition 3.6.4]{benson}); the proof over $\Z_p$ works mutatis mutandis) and so  by Proposition \ref{perm modules}(ii) can be assumed to be contained in $M_p$. Hence $M_1\cap F^{ab}=0$ and so the bottom map of the diagram is injective on $F^{ab}$. Therefore the upper map of the diagram is injective on $F$ and the lemma is proved in this case.

   Suppose $|K|>p$  and   let $H$ be  a subgroup of index $p$ in $K$.   Put $\tilde F_H= \langle X_i\mid 1\neq A_i\leq H, i\in I_0\rangle^{\tilde F}$. By the induction hypothesis applied to $\tilde F H$ we know that
   $F\cap \tilde F_H=1$ and so   the natural epimorphism $\tilde G\longrightarrow \tilde G/\tilde F_H$ restricted to $G$ is injective. Let $I_H=\{i\in I_0\mid A_i\leq H\}$. It follows from the presentation of a special HNN-extension  that $\tilde G/\tilde F_H={\rm HNN}(K, A_i, X_i, i\in I\setminus I_H\}$ and so satisfies the premises of the lemma.    So if $I_H\neq \emptyset$ for  such $H$ then we deduce the result from the induction hypothesis.

   Thus we left with the case when $|I_0|=1$ and $A_{i_0}=K$ for $i_0\in I_0$. In this case $\tilde G=(K\times F(X_{i_0},*))\amalg F_0$ for some free pro-$p$ group $F_0$.
     Let $C$ be a normal subgroup of order $p$ of $K$. Then $C_{\tilde F} (C)=C_{\tilde F}(K)=F(X_{i_0},*)$ (see \cite[Theorem 9.1.12]{RZ}). So $\tilde F C$ satisfies the premises of the lemma and the result follows from the induction hypothesis applied to $\tilde F C$ in this case.

  \end{proof}

\begin{thm}\label{embedding semidirect} Let $G=F\rtimes H$ be a semidirect product of a free pro-$p$ group $F$ and a finite $p$-group $H$. Suppose $G$ have finite centralizers of the non-trivial torsion elements. Then $G$ embeds into a free pro-$p$ product $H\amalg F_0$ of $H$ and a free pro-$p$ group $F_0$.\end{thm}

\begin{proof}    Consider $G$ as a subgroup of an HNN-extension $$\tilde G=\HNN(G,A_i,\phi_i, X_i), i\in I$$ (Lemma \ref{mono}).  By Corollary \ref{permext} $\tilde G$ is a special HNN-extension $\tilde G={\rm HNN}(H, H_e, X_e), e\in E$.

\smallskip
Since $G$ has finite centralizers of non-trivial torsion elements $C_F(A_i)=1$ for every $i$, so applying Lemma \ref{centralizers out}  we obtain the natural epimorphism $\tilde G\longrightarrow\tilde G/\langle X_i)\mid A_i\neq 1\rangle^{\tilde G}=F_0\amalg H$   whose restriction to $G$ is an injection. The result follows.
\end{proof}

\begin{cor} \label{free-by-finite embedding}
Let $G$ be a  pro-$p$ group possessing an open normal free pro-$p$ group $F$ and having finite centralizers of the non-trivial torsion elements. Then $G$ embeds into a free pro-$p$ product $G/F \amalg F_1$ of the finite quotient group $G/F$ and a free pro-$p$ group $F_1$.
\end{cor}

\begin{proof} Embed $G$ into $G_0=G\amalg G/F=F_0\rtimes G/F$ with $F_0$   free pro-$p$  containing $F$ and apply Theorem \ref{embedding semidirect}. \end{proof}

We are ready to prove Theorem \ref{finite centralizers embedding}.

\begin{thm}\label{virtually free embedding}  Let $G$ be a virtually free pro-$p$ group having finite centralizers of the non-trivial torsion elements. Then $G$ embeds into a free pro-$p$ product $G_0=F_0\amalg H$ of a finite $p$-group $H$ and a free pro-$p$ group $F_0$.

\end{thm}

\begin{proof} Let $F$ be a core of an open free pro-$p$ subgroup of $G$. Then $F$ is open normal and we can apply Corollary \ref{free-by-finite embedding}.\end{proof}

Let $G$ be a pro-$p$ group and $\{G_x\mid x\in X\}$ be a family of subgroups indexed by a profinite space $X$. Following \cite[Section 5.2]{R 2017} we say that $\{G_x\mid x\in X\}$ is continuous if for any open subgroup $U$ of $G$ the subset $\{ x\in X\mid G_x\subseteq U\}$ is open.

\begin{lem}\label{continuous family} Let $G$ be a virtually free pro-$p$ group having finite centralizers of the non-trivial torsion elements. Then the maximal finite subgroups of $G$  can be indexed by some  profinite $G$-space  $X$ such that  the family
$\mathcal F=\{G_x\mid x\in X\}$ of them is continuous. \end{lem}

\begin{proof} Let $G_0=F_0\amalg H$ be a free pro-$p$ product containing $G$ from Theorem \ref{virtually free embedding}. Consider  the family $\{G_x\mid x\in G_0/H\}$ of the stabilizers of points of the profinite space $G_0/H$ on which $G$ acts from the left. By \cite[4.8 (3-d paragraph)]{Me} or \cite[Lemma 5.2.2] {R 2017}) $\{G_x\mid x\in G_0/H\}$ is a continuous family of subgroups of $G$. Since $G\cap F_0$ is an open free pro-$p$ subgroup of $G$ the subset $\{x\in G_0/H\mid  G_x\leq F_0\cap G\}$ is open and clearly $G_x\not\leq F_0\cap G$ iff $G_x\neq \{1\}$. Therefore the subset $X=\{x\in G_0/H\mid  G_x\neq \{1\}\}$ is closed and so one deduces from the definition of a continuous family that $\mathcal F=\{G_x\mid x\in X\}$ is a continuous family of subgroups of $G$. Clearly $X$ is $G$-invariant, and since $G_{0x}=xHx^{-1}$ is a maximal finite subgroup of $G_0$, $G_x$ is a maximal finite subgroup of $G$. The lemma is proved.

\end{proof}

We finish the section with the proof of Theorem \ref{decomposition}.

\begin{thm} Let $G$ be a virtually free pro-$p$ group having finite centralizers of non-trivial torsion elements and $X$ be the space from Lemma \ref{continuous family}. Then $G$ is a free pro-$p$ product of finite $p$-groups and a free pro-$p$ group if and only if  the natural quotient map $\theta:X\longrightarrow X/G$ with respect to the action of $G$ admits a continuous section.\end{thm}

\begin{proof} `If' follows from the pro-$p$ version of the Kurosh Subgroup Theorem  (\cite[Thm. 4.3]{Me}) applied to the subgroup $G$ of $G_0$ taking into account that the existence of a continuous section $s$ to $\theta$ is only what is needed to remove the second countability hypothesis from its statement. See for example \cite[Comment (5.1)]{Me} , where this  is explicitely written.  Alternatively one can use  \cite[Thm 9.6.1 (a)]{R 2017}) by considering a natural action of $G$ on the standard pro-$p$ tree $T$ for $G_0$ and as in the proof of (b) of this theorem one can use \cite[Lemma 5.2.2]{R 2017}) and continuity of $s$ to show that $\{G_x\mid x\in s(X/G)\}$ is a continuous family.

`Only if'. Let $G=\coprod_{t\in T}G_t \amalg F_0$, where $\{G_t\mid t\in T\}$ is a continuous family of finite subgroups of $G$.  Let $F$ be an open free pro-$p$ subgroup of $G$.  By \cite[(1.2) Remark]{Me} or \cite[Lemma 5.2.1]{R 2017} $E=\bigcup_{t\in T} G_t$ is closed in $G$ and so $E_0=E\setminus F=E\setminus \{1\}$ is closed, where $F$ is an open free pro-$p$ subgroup of $G$. Consider a continuous map $\eta:G\times X\longrightarrow X\times X$, $\eta(g,x)=(gx,x)$. Let $X_0$ be the projection of $\eta(E_0)\cap D$ on the first coordinate, where $D$ is the diagonal of $X\times X$. It suffices to show that the restriction of $\theta$  to $X_0$ is a homeomorphism.  Clearly, $\theta_{|X_0}$ is continuous, so we need to check bijectivity. Choose $x_1\neq x_2\in X_0$. Then $\eta^{-1}(x_1, x_1)=G_{t_1}\times \{x_1\}$ and $\eta^{-1}(x_2, x_2)=G_{t_2}\times \{x_2\}$ for some $t_1\neq t_2 \in T$.  Since $G_{t_1}$ and $G_{t_2}$ are not conjugate in $G$ (see \cite[Cor. 7.1.5 (a)]{R 2017}) they stabilize points $x_1$ and $x_2$ in different  $G$-orbits, so $\theta_{|X_0}$ is injective. The surjectivity of $\theta_{|X_0}$ is clear, since any finite group $G_x$ is conjugate to some $G_t$ in $G$ see \cite[Cor. 7.1.3]{R 2017}) and $G_t$ stabilizes some point $x_0\in X_0$, i.e. $\theta(x)=\theta(x_0)$.

\end{proof}

\section{Appendix}

We shall give a construction from \cite{HZ 07} that uses an HNN-extension to embed a semidirect product $G=F\rtimes H$ of a free pro-$p$ group $F$ and a finite $p$-group $H$ into a semidirect product $\tilde G=\tilde F\rtimes H$ of a free pro-$p$ group $\tilde F$ and the same group $H$ such that all finite subgroups of $\tilde G$ conjugate into $H$. In particular it gives another proof of Proposition 1.3 in \cite{crelle}.

 Let  $\pi:G\longrightarrow H$ be the natural projection; note that $\pi$ restricted to any finite subgroup of $G$ is an injection. Let $Fin(G)$ be the set of all finite subgroups of $G$. Since the order of any finite subgroup of $G$ does not exceed $|H|$  and $G$ is a projective
limit of finite groups, $\Fin G$ is the projective limit of the
respective sets of  subgroups of order $\leq |H|$ -- hence it carries a natural
topology (the {\it subgroup topology}) -- turning it into a profinite
space. Equipped with this topology, $\Fin G$ with $G$ acting by
conjugation becomes a profinite $G$-space.

Let ${\caA}=\{A_i, i\in I\}$ be the set of all subgroups of $H$. The projection $\pi$ induces a continuous surjection $\rho: Fin(G)\longrightarrow {\caA}$. Put $X_i=\rho^{-1}(\{A_i\})$. Define $\phi_i:A_i\times X_i\longrightarrow G$ by setting $\phi_i(a,x)$ to be the unique element $a_x$ of the group $x\in X_i$ such that $\pi(a_x)=a$.  Form an HNN-extension $\tG:=\HNN(G,A_i,\phi_i, X_i)$, $i\in I$ (since $X_i$ are compact, the distinguished point from Definition \ref{HNN-ext} can be omitted). By \cite[Theorem 7.1.2]{R 2017} (cf. also \cite[Example 6.2.3 (e)]{R 2017}) each finite subgroup of $\tilde G$  is conjugate into $G$ and hence by construction into $H$. Note also that the natural epimorphism $\pi:G\longrightarrow H$ extends by the universal property to $\tilde \pi:\tilde G\longrightarrow H$. By \cite[Lemma 10]{HZ 07} or \cite[Theorem 9.6.1 (a)]{R 2017} $\tilde F=ker(\tilde pi)$ is a free pro-$p$ products of conjugates of $\tilde F\cap G=F$ and a free pro-$p$ group; thus $\tilde F$ is free pro-$p$.

\begin{lem}\label{mono}
The natural homomorphism $G\longrightarrow \tilde G$ is an injection. \end{lem}

\begin{proof} Let $\tG^{abs}=\HNN^{abs}(G,A_i,\phi_i, X_i)$, $i\in I$ be the abstract HNN-extension. Then it suffices to show that the natural homomorphism $\tG^{abs}\longrightarrow \tilde G$ is injective. By Definition \ref{HNN-ext} $\tG:=\HNN(G,A_i,\phi_i, X_i)$, $i\in I$
is defined to be  the quotient of $G\amalg F(\bigcup_{i\in I} X_i)$, where $F(\bigcup_{i\in I} X_i)$ is a free pro-$p$ group on  $(\bigcup_{i\in I} X_i)$, modulo the
relations $\phi_i(a_i)=x_ia_ix_i^{-1}$ for all $x_i\in X_i$, $i\in I$. Thus we have the following commutative diagram

$$\xymatrix{ G \star F(\bigcup_{i\in I} X_i)\ar[r]\ar[d]& G\amalg F(\bigcup_{i\in I} X_i)\ar[d]\\
\tG^{abs}\ar[r] & \tG},$$ where $\star$ means the abstract free product.

Note that $G\amalg F(\bigcup_{i\in I} X_i)$ is  the pro-$p$ completion of $G * F(\bigcup_{i\in I} X_i)$ with respect to the family of normal subgroups $N$ of finite index such that $N\cap G$ is open in $G$ and $(\bigcup_{i\in I} X_i)\longrightarrow (\bigcup_{i\in I} X_i)N/N$ is continuous. Therefore $\tG$ is the pro-$p$ completion of $\tG^{abs}$ with respect to the family of images of $N$ in $\tG^{abs}$, i.e. with respect to the family of normal subgroups $U$ of finite index in $\tG^{abs}$ satisfying the same properties. Moreover, w.l.o.g we may assume that $U\cap G\leq F$.

Choose   an open normal subgroup $V\triangleleft_o G$ of $G$ contained in $F$. Then the natural epimorphism $G\longrightarrow G/V$ induces the natural continuous map $\rho_V:Fin(G)\longrightarrow Fin(G/V)$ and therefore continuous  maps $\rho_{V,i}: X_i\longrightarrow Fin(G/V)$.
  Then by the universal property of abstract HNN-extensions $\rho_{V,i}$ and $G\longrightarrow G/V$ extend to an epimorphism $$\tilde\rho_V:\tG^{abs}\longrightarrow {\rm HNN}^{abs}(G/V, A_iV/V, \rho_V(X_i)).$$  Note that the natural epimorphism $\pi:G\longrightarrow H$ extends to the epimorphism $\tilde\pi^{abs}:\tilde G^{abs}\longrightarrow H$ (the samer way as to $\tilde \pi$). Moreover, $\tilde\pi^{abs}$ factors through $\tilde\rho_V$, i.e. we have the natural epimorphism $$\varphi_V:{\rm HNN}^{abs}(G/V, A_iV/V, \rho_V(X_i))\longrightarrow H$$ whose kernel is an abstract free product  of $F/V$ and an abstract free group (by the Kurosh subgroup theorem). Moreover, since $V\leq F$, ${\rm HNN}^{abs}(G/V, A_iV/V, \rho_V(X_i))=ker(\varphi_V)\rtimes HV/V$ and hence is residually $p$ (as an extension of a residually $p$ group by a finite $p$-group). Therefore there exists an epimorphism
   $$\varphi_{VK}:{\rm HNN}^{abs}(G/V, A_iV/V, \rho_V(X_i))\longrightarrow K$$  to a finite $p$-group $K$ injective on $G/V$. Then for the kernel $U_{VK}$ of $\varphi_{VK}\tilde\rho_V$ we have $U_{VK}\cap G$ is open in $G$ and $(\bigcup_{i\in I} X_i)\longrightarrow (\bigcup_{i\in I} X_i)U_{VK}/U_{VK}$ is continuous. Hence $\bigcap_{V}(U_{VK}\cap G)=1$ implies that $G$ indeed embeds in $\tilde G$.

\end{proof}

\end{document}